\documentclass[12pt,reqno]{article}

\usepackage[usenames]{color}
\usepackage{amssymb}
\usepackage{graphicx}
\usepackage{amscd}
\usepackage[colorlinks=true,
linkcolor=webgreen,
filecolor=webbrown,
citecolor=webgreen]{hyperref}

\definecolor{webgreen}{rgb}{0,.5,0}
\definecolor{webbrown}{rgb}{.6,0,0}

\usepackage{color}
\usepackage{fullpage}
\usepackage{float}

\usepackage{graphics,amsmath,amssymb}
\usepackage{amsthm}
\usepackage{amsfonts}
\usepackage{latexsym}
\usepackage{epsf}

\usepackage{verbatim}

\usepackage{tikz}
\usepackage{forest}

\usepackage{mathtools}

\usepackage[blocks]{authblk}

\theoremstyle{plain}
\newtheorem{theorem}{Theorem}
\newtheorem{corollary}[theorem]{Corollary}

\newtheorem{proposition}[theorem]{Proposition}

\theoremstyle{definition}

\theoremstyle{remark}

\begin{document}
\title{Fun with Latin Squares}
\author{Michael Han}
\author{Ella Kim}
\author{Evin Liang}
\author{Miriam (Mira) Lubashev}
\author{Oleg Polin}
\author{Vaibhav Rastogi}
\author{Benjamin Taycher}
\author{Ada Tsui}
\author{Cindy Wei}
\affil{PRIMES STEP}
\author{Tanya Khovanova}
\affil{MIT}

\maketitle

\begin{abstract}
Do you want to know what an anti-chiece Latin square is? Or what a non-consecutive toroidal modular Latin square is? We invented a ton of new types of Latin squares, some inspired by existing Sudoku variations. We can't wait to introduce them to you and answer important questions, such as: do they even exist? If so, under what conditions? What are some of their interesting properties? And how do we generate them?
\end{abstract}

\section{Introduction}\label{sec:intro}

In this paper, we study variations of Latin squares. 

A Latin square is an $n$ by $n$ table filled with integers 1 through $n$ so that each row and column contain each number exactly once. In other words, each row/column is a permutation of integers 1 through $n$, and $n$ is called the order, or the size of the Latin square. The definition and history of Latin squares are covered in many books; see, for example \cite{CD,WG}.

The number of Latin squares of order $n$ was calculated by Shao and Wai \cite{SW}. It is sequence A002860 in the OEIS \cite{OEIS}:
\[1,\ 2,\ 12,\ 576,\ 161280,\ 812851200,\ 61479419904000,\ 108776032459082956800, \ldots.\]

Special Latin squares have been studied before. A strict knight's move Latin square is a Latin square in which any two cells that contain the same symbol are joined by a sequence of knight's moves using only cells containing that symbol. Owens \cite{O} showed that there exists a strict knight's move Latin square of order $n$ without diagonal adjacencies if and only if $n = 5$. To within permutations of the symbols and transposition, there is only one such square which is shown on the right of Figure~\ref{fig:skm5}. He also showed that allowing diagonal adjacencies, there exists a strict knight's move Latin square of order $n$ if and only if either $n = 5$ or $n \equiv 0$ (mod 4). 

Figure~\ref{fig:skm5} shows two examples of strict knight's move Latin squares of order 5. The square on the left contains diagonal adjacencies, while the square on the right does not.

\begin{figure}[ht!]
{
\begin{center}
\begin{tikzpicture}
\draw[step=0.5cm,color=black, line width=1.5] (-1,-1) grid (1.5,1.5);
\node at (-0.75, +1.25) {1};
\node at (-0.25, +1.25) {2};
\node at (+0.25, +1.25) {3};
\node at (+0.75, +1.25) {4};
\node at (+1.25, +1.25) {5};
\node at (-0.75,+0.75) {5};
\node at (-0.25,+0.75) {4};
\node at (+0.25,+0.75) {1};
\node at (+0.75,+0.75) {3};
\node at (+1.25,+0.75) {2};
\node at (-0.75,+0.25) {4};
\node at (-0.25,+0.25) {3};
\node at (+0.25,+0.25) {2};
\node at (+0.75,+0.25) {5};
\node at (+1.25,+0.25) {1};
\node at (-0.75, -0.25) {2};
\node at (-0.25, -0.25) {5};
\node at (+0.25, -0.25) {4};
\node at (+0.75, -0.25) {1};
\node at (+1.25, -0.25) {3};
\node at (-0.75,-0.75) {3};
\node at (-0.25,-0.75) {1};
\node at (+0.25,-0.75) {5};
\node at (+0.75,-0.75) {2};
\node at (+1.25,-0.75) {4};
\end{tikzpicture}
\quad \quad
\begin{tikzpicture}
\draw[step=0.5cm,color=black, line width=1.5] (-1,-1) grid (1.5,1.5);
\node at (-0.75, +1.25) {1};
\node at (-0.25, +1.25) {2};
\node at (+0.25, +1.25) {3};
\node at (+0.75, +1.25) {4};
\node at (+1.25, +1.25) {5};
\node at (-0.75,+0.75) {4};
\node at (-0.25,+0.75) {5};
\node at (+0.25,+0.75) {1};
\node at (+0.75,+0.75) {2};
\node at (+1.25,+0.75) {3};
\node at (-0.75,+0.25) {2};
\node at (-0.25,+0.25) {3};
\node at (+0.25,+0.25) {4};
\node at (+0.75,+0.25) {5};
\node at (+1.25,+0.25) {1};
\node at (-0.75, -0.25) {5};
\node at (-0.25, -0.25) {1};
\node at (+0.25, -0.25) {2};
\node at (+0.75, -0.25) {3};
\node at (+1.25, -0.25) {4};
\node at (-0.75,-0.75) {3};
\node at (-0.25,-0.75) {4};
\node at (+0.25,-0.75) {5};
\node at (+0.75,-0.75) {1};
\node at (+1.25,-0.75) {2};
\end{tikzpicture}
\end{center}
	\caption{Strict knight's move Latin squares of order 5.}
	\label{fig:skm5}
	}
\end{figure}

We start with definitions in Section~\ref{sec:def}. 

Section~\ref{sec:NC} is devoted to non-consecutive Latin squares, which we define as Latin squares in which orthogonally adjacent cells can't be consecutive. We alternatively call them shy squares. Not counting a trivial non-consecutive Latin square of order 1, the smallest possible such square is of order 5. Any non-consecutive Latin square of order 5 is built by cycling a cyclable permutation. It follows that all such squares of order 5 are also toroidal and modular. Such squares of larger orders are easier to construct, and they exist. We show how to construct some of them by cycling a permutation or by using a smaller square. We also introduce non-consecutive by king's move Latin squares in Section~\ref{sec:NonconbyKM}. We show that they exist for $n > 6$. If we add an extra constraint that diagonal neighbors can't be the same numbers, then such squares exist for $n > 12$.

Section~\ref{sec:Consecutive} is devoted to Latin squares on the other side of the spectrum from the previous one: the neighboring numbers have to be consecutive. We also call such squares \textit{nosy}. This constraint is seriously restrictive. We show that consecutive as well as toroidal consecutive Latin squares do not exist for $n>2$. Modular consecutive Latin squares exists, and they are left- or right- cyclic for any $n$ except $n=4$. There are $4n$ of consecutive modular Latin squares of order $n$: there are $2n$ ways to choose the first row, and there are two ways to choose which way to cycle it. For $n=4$, there are two more consecutive Latin squares, neither of which is cyclic. In any case, all modular cyclic Latin squares are also toroidal.

In Section~\ref{sec:AC} we introduce a \textit{chiece} which means a generic chess piece. The word chiece is a portmanteau of the two words chess and piece. This section is devoted to anti-chiece Latin squares and is inspired by the famous Sudoku type: anti-knight. In anti-chiece Latin squares, cells separated by a chiece's move in chess have to be different. We start the section by studying anti-knight Latin squares. The smallest non-trivial example of such a square is of order 4, and there are 4 anti-knight Latin squares with a fixed first row. Surprisingly, for order 5, there are only 2 such squares with a fixed first row. Moreover, all non-consecutive Latin squares of order 5 are also anti-knight squares. Anti-knight Latin squares exist for any order greater than 3. 

The first non-trivial example of an anti-queen Latin square is of order 5: there are two squares with a fixed first row. For larger orders $n$, it is easy to produce such squares by using a $k$-cyclic square, where $n$ is coprime with $k$, $k-1$, and $k+1$. Such $k$-$n$ pair exists whenever $n$ is coprime with 6.

Anti-queen Latin squares are also anti-king. We show that anti-king but not anti-queen Latin squares exist for composite $n > 6$.

Section~\ref{sec:Chess} is devoted to chiece Latin squares. A chiece Latin square is such that every cell has to have at least one other cell with the same number at a chiece's move apart. We show that if there exists an $n$ by $n$ chiece Latin square, then for all positive integers $m$, there exists an $nm$ by $nm$ chiece Latin square. It follows that bishop and king Latin squares exist for all even orders, and knight Latin squares exist for all orders divisible by four or five. We also show that king Latin squares exist only for even orders.

\section{Definitions and such}\label{sec:def}

A Latin square is an $n$ by $n$ table filled with integers 1 through $n$ so that each number appears exactly once in each row and column. The number $n$ can be referred to as the \textit{size} or the \textit{order} of the Latin square. In colloquial language, the word size is typically used, but the word order is more common in formal mathematical texts. As this paper is recreational, we use them interchangeably.
	
This is a good place to provide an example of a Latin square, but there are so many to choose from. As such, a systematic method to choose a square is to arrange squares as words in a dictionary. To do this, we consider the integers 1 through $n$ to be the letters of an alphabet, ordered from 1 to $n$. Then, we can view a Latin square as a word in this alphabet written by rows top to bottom and within each row left to right. The \textit{lexicographically earliest} Latin square with given constraints is defined to be the first such square when the squares are sorted in lexicographic order. For example, the lexicographically earliest Latin square of size 4 is shown in Figure~\ref{fig:LE4}.

\begin{figure}[ht!]
	{
	\begin{center}
	\begin{tikzpicture}
	\draw[step=0.5cm,color=black, line width=1.5] (-1,-1) grid (1,1);
	\node at (-0.75,+0.75) {1};
	\node at (-0.25,+0.75) {2};
	\node at (+0.25,+0.75) {3};
	\node at (+0.75,+0.75) {4};
	\node at (-0.75,+0.25) {2};
	\node at (-0.25,+0.25) {1};
	\node at (+0.25,+0.25) {4};
	\node at (+0.75,+0.25) {3};
	\node at (-0.75, -0.25) {3};
	\node at (-0.25, -0.25) {4};
	\node at (+0.25, -0.25) {1};
	\node at (+0.75, -0.25) {2};
	\node at (-0.75,-0.75) {4};
	\node at (-0.25,-0.75) {3};
	\node at (+0.25,-0.75) {2};
	\node at (+0.75,-0.75) {1};
	\end{tikzpicture}
	\end{center}
	\caption{The lexicographically earliest Latin square of order 4.}
	\label{fig:LE4}
	}
\end{figure}

\subsection{Cyclicity}

Arguably the easiest method to build a Latin square is to cycle its first row such that every subsequent row is a cyclic permutation of it. 

Figure~\ref{fig:LC4} shows a general example, where the shift from one row to another is not the same. In this example, the second row is a shift by 2 to the right of the first row, the third row is a shift to the left of the first row, and the last row is a shift to the right of the first row.	
		
	\begin{figure}[ht!]	
		{	
		\begin{center}	
		\begin{tikzpicture}	
		\draw[step=0.5cm,color=black, line width=1.5] (-1,-1) grid (1,1);	
		\node at (-0.75,+0.75) {1};	
		\node at (-0.25,+0.75) {2};	
		\node at (+0.25,+0.75) {3};	
		\node at (+0.75,+0.75) {4};	
		\node at (-0.75,+0.25) {3};	
		\node at (-0.25,+0.25) {4};	
		\node at (+0.25,+0.25) {1};	
		\node at (+0.75,+0.25) {2};	
		\node at (-0.75, -0.25) {2};	
		\node at (-0.25, -0.25) {3};	
		\node at (+0.25, -0.25) {4};	
		\node at (+0.75, -0.25) {1};	
		\node at (-0.75,-0.75) {4};	
		\node at (-0.25,-0.75) {1};	
		\node at (+0.25,-0.75) {2};	
		\node at (+0.75,-0.75) {3};	
		\end{tikzpicture}	
		\end{center}	
		\caption{A Latin square of order 4, in which rows are cyclic permutations of each other.}	
		\label{fig:LC4}	
		}	
	\end{figure}

We call a Latin square \textit{k-cyclic} if, for some $k$, each subsequent row is a right cyclic shift by $k$ of the previous row. Sometimes we omit $k$ and call a Latin square \textit{cyclic} if it is $k$-cyclic for some $k$.
	
The cycling procedure doesn't always produce a Latin square: if $k$ is not coprime with $n$, then the resulting square has repeated numbers in each column. However, if $k$ is coprime with $n$, then the cycling procedure does produce a Latin square. Thus, often the easiest choice for a shift is $k=1$, as 1 is coprime with every number.
	
We call a Latin square \textit{left-cyclic} if each subsequent row is a left shift by one of the previous row, and similarly, we call a Latin square \textit{right-cyclic} if each subsequent row is a right shift by one. In linear algebra, such matrices are often called anit-circulant and circulant respectively. Examples of left- and right-cyclic Latin squares are in Figure~\ref{fig:3}.

\begin{figure}[ht!]
	{
	\begin{center}
	\begin{tikzpicture}
	\draw[step=0.5cm,color=black, line width=1.5] (-1,-1) grid (0.5,0.5);
	\node at (-0.75,+0.25) {1};
	\node at (-0.25,+0.25) {2};
	\node at (+0.25,+0.25) {3};
	\node at (-0.75, -0.25) {2};
	\node at (-0.25, -0.25) {3};
	\node at (+0.25, -0.25) {1};
	\node at (-0.75,-0.75) {3};
	\node at (-0.25,-0.75) {1};
	\node at (+0.25,-0.75) {2};
	\end{tikzpicture}
\quad \quad
	\begin{tikzpicture}
	\draw[step=0.5cm,color=black, line width=1.5] (-1,-1) grid (0.5,0.5);
	\node at (-0.75,+0.25) {1};
	\node at (-0.25,+0.25) {2};
	\node at (+0.25,+0.25) {3};
	\node at (-0.75, -0.25) {3};
	\node at (-0.25, -0.25) {1};
	\node at (+0.25, -0.25) {2};
	\node at (-0.75,-0.75) {2};
	\node at (-0.25,-0.75) {3};
	\node at (+0.25,-0.75) {1};
	\end{tikzpicture}
	\end{center}
	\caption{A left- and a right-cyclic Latin square of order 3.}
	\label{fig:3}
	}
\end{figure}

Interestingly, all Latin squares of size 3 are cyclic.

\subsection{Toroidal Latin squares}

In a toroidal Latin square, the square is considered to be a torus. That is, the left-most and right-most cells of the same row are considered to be adjacent, and similarly, the top-most and bottom-most cells of the same column are considered to be adjacent. So why do we care if the square is a torus or not? In standard Latin squares, this definition doesn't matter, but in this paper, we add restrictions to Latin squares, and this definition becomes relevant.
	
In Section~\ref{sec:NC} and Section~\ref{sec:Consecutive}, we consider non-consecutive and consecutive Latin squares in which adjacent numbers are either not allowed to or required to be consecutive. Toroidal Latin squares are more restrictive for non-consecutive and consecutive Latin squares as they add extra constraints for the border cells.

In Section~\ref{sec:AC} and Section~\ref{sec:Chess}, we look at how each number in a Latin square behaves with the relationship to some chess moves. In anti-chess Latin squares, we forbid the numbers to be a chess move apart for a particular chess piece. This makes toroidal Latin squares more restrictive as pieces can move through the border lines. In chess Latin squares, we require that each number is a chess move apart from another cell with the same number for a particular chess piece. This makes toroidal Latin squares less restrictive: a chess Latin square is automatically a toroidal chess Latin square.

\subsection{Modular Latin squares}

In a \textit{modular Latin square}, numbers are considered modulo $n$. That is, 1 and $n$ are considered consecutive numbers. This is only relevant in Section~\ref{sec:NC} where we study non-consecutive Latin squares and in Section~\ref{sec:Consecutive} where we study consecutive Latin squares. Modular Latin squares are more restrictive for non-consecutive Latin squares as they add an extra constraint forbidding numbers 1 and $n$ to be neighbors. Modular Latin squares are less restrictive for consecutive Latin squares as they allow numbers 1 and $n$ to be neighbors.

\subsection{Products of Latin squares}

If we have two Latin squares of orders $m$ and $n$, there is a simple method, known as the Kronecker product, to generate a Latin square of order $mn$.
	
We denote an entry with coordinates $(i,j)$ in our $m$ by $m$ Latin square $L$ by $L_{i,j}$ and in our $n$ by $n$ Latin square $B$ by $B_{i,j}$. 

We divide the $mn$ by $nm$ grid into $m^2$ smaller grids with dimensions $n$ by $n$. In each smaller $n$ by $n$ grid, the Latin square $B$ is placed with each entry incremented by $n$ times the value in the corresponding cell of the $m$ by $m$ Latin square $L-1$. More formally, the $mn$ by $mn$ grid is filled in by placing the number $(L_{a,c}-1)n+B_{b,d}$ into the cell with coordinates $(an+b,cn+d)$. 
	
For example, suppose the Latin squares in Figure~\ref{fig:3and2} are the Latin squares we seek to multiply.

\begin{figure}[ht!]
{
\begin{center}
\begin{tikzpicture}
\draw[step=0.5cm,color=black, line width=1.5] (-1,-1) grid (0.5,0.5);
\node at (-0.75,+0.25) {1};
\node at (-0.25,+0.25) {2};
\node at (+0.25,+0.25) {3};
\node at (-0.75,-0.25) {2};
\node at (-0.25,-0.25) {3};
\node at (+0.25,-0.25) {1};
\node at (-0.75, -0.75) {3};
\node at (-0.25, -0.75) {1};
\node at (+0.25, -0.75) {2};
\end{tikzpicture}
\quad \quad
\begin{tikzpicture}
\draw[step=0.5cm,color=black, line width=1.5] (0,0) grid (1,1);
\node at (+0.25,+0.25) {2};
\node at (+0.75,+0.25) {1};
\node at (+0.25,+0.75) {1};
\node at (+0.75,+0.75) {2};
\end{tikzpicture}
\end{center}
	\caption{The two Latin squares to multiply.}
	\label{fig:3and2}
	}
\end{figure}
The product of the first Latin square times the second Latin square is shown on the left of Figure~\ref{fig:product}, while the product of the second Latin square times the first Latin square is shown on the right. Note that these two products are different due to this multiplication not being commutative.
\begin{figure}[ht!]
{
\begin{center}
\begin{tikzpicture}
\draw[step=0.5cm,color=black, line width=1.5] (0,0) grid (3,3);
\node at (+0.25,+0.25) {6};
\node at (+0.75,+0.25) {5};
\node at (+1.25,+0.25) {2};
\node at (+1.75,+0.25) {1};
\node at (+2.25,+0.25) {4};
\node at (+2.75,+0.25) {3};
\node at (+0.25,+0.75) {5};
\node at (+0.75,+0.75) {6};
\node at (+1.25,+0.75) {1};
\node at (+1.75,+0.75) {2};
\node at (+2.25,+0.75) {3};
\node at (+2.75,+0.75) {4};
\node at (+0.25,+1.25) {4};
\node at (+0.75,+1.25) {3};
\node at (+1.25,+1.25) {6};
\node at (+1.75,+1.25) {5};
\node at (+2.25,+1.25) {2};
\node at (+2.75,+1.25) {1};
\node at (+0.25,+1.75) {3};
\node at (+0.75,+1.75) {4};
\node at (+1.25,+1.75) {5};
\node at (+1.75,+1.75) {6};
\node at (+2.25,+1.75) {1};
\node at (+2.75,+1.75) {2};
\node at (+0.25,+2.25) {2};
\node at (+0.75,+2.25) {1};
\node at (+1.25,+2.25) {4};
\node at (+1.75,+2.25) {3};
\node at (+2.25,+2.25) {6};
\node at (+2.75,+2.25) {5};
\node at (+0.25,+2.75) {1};
\node at (+0.75,+2.75) {2};
\node at (+1.25,+2.75) {3};
\node at (+1.75,+2.75) {4};
\node at (+2.25,+2.75) {5};
\node at (+2.75,+2.75) {6};
\end{tikzpicture}
\quad \quad
\begin{tikzpicture}
\draw[step=0.5cm,color=black, line width=1.5] (0,0) grid (3,3);
\node at (+0.25,+0.25) {6};
\node at (+0.75,+0.25) {4};
\node at (+1.25,+0.25) {5};
\node at (+1.75,+0.25) {3};
\node at (+2.25,+0.25) {1};
\node at (+2.75,+0.25) {2};
\node at (+0.25,+0.75) {5};
\node at (+0.75,+0.75) {6};
\node at (+1.25,+0.75) {4};
\node at (+1.75,+0.75) {2};
\node at (+2.25,+0.75) {3};
\node at (+2.75,+0.75) {1};
\node at (+0.25,+1.25) {4};
\node at (+0.75,+1.25) {5};
\node at (+1.25,+1.25) {6};
\node at (+1.75,+1.25) {1};
\node at (+2.25,+1.25) {2};
\node at (+2.75,+1.25) {3};
\node at (+0.25,+1.75) {3};
\node at (+0.75,+1.75) {1};
\node at (+1.25,+1.75) {2};
\node at (+1.75,+1.75) {6};
\node at (+2.25,+1.75) {4};
\node at (+2.75,+1.75) {5};
\node at (+0.25,+2.25) {2};
\node at (+0.75,+2.25) {3};
\node at (+1.25,+2.25) {1};
\node at (+1.75,+2.25) {5};
\node at (+2.25,+2.25) {6};
\node at (+2.75,+2.25) {4};
\node at (+0.25,+2.75) {1};
\node at (+0.75,+2.75) {2};
\node at (+1.25,+2.75) {3};
\node at (+1.75,+2.75) {4};
\node at (+2.25,+2.75) {5};
\node at (+2.75,+2.75) {6};
\end{tikzpicture}
\end{center}
	\caption{The products of two Latin squares.}
	\label{fig:product}
	}
\end{figure}

\section{Non-consecutive (Shy) Latin squares}\label{sec:NC}

A \textit{non-consecutive Latin square} is a Latin square where two orthogonally adjacent cells can't contain consecutive numbers. Consecutively, each row/column is a non-consecutive permutation. We also call such squares \textit{shy} Latin squares as numbers avoid their neighbors.

We start by looking at non-consecutive permutations.

\subsection{Non-consecutive Permutations}

A permutation is called \textit{non-consecutive} (also \textit{shy}) if any two neighboring numbers in it are not consecutive. For example, the permutation 3241 is not shy because 3 and 2 are next to each other, but the permutation 2413 is shy.

\subsubsection{Counting non-consecutive permutations}

Consider the sequence $a(n)$: the number of non-consecutive permutations of size $n$. We have $a(1) = 1$. We know that we can't place 2 next to either 1 or 3. Hence, $a(2) = a(3)=0$. 

For $n=4$, the number 2 can't have two neighbors. Similarly, 3 can't have two neighbors. It follows that 2413 and 3142 are the only two non-consecutive permutations of size 4.

Consider $n=5$. We can see that the lexicographically earliest non-consecutive permutation is 13524. The first and the last terms are not neighbors, so cycling this permutation produces four more non-consecutive permutations: 35241, 52413, 24135, and 41352. We can add five more permutations that are reverses of the above. They are also cycles of each other: 42531, 14253, 31425, 53142, and 25314. We have four more permutations: 24153, 31524, 35142, and 42513. These permutations can't cycle as the first and the last numbers are neighbors. We call these permutations \textit{non-cyclable}.

We calculated that there are 14 non-consecutive permutations of size 5.

The sequence $a(n)$ is in the OEIS as integer sequence A002464:

\[1,\ 0,\ 0,\ 2,\ 14,\ 90,\ 646,\ 5242,\ 47622,\ 479306,\ 5296790,\ 63779034, \ldots.\]

This sequence also counts the number of ways to arrange $n$ non-attacking kings on an $n$ by $n$ board with one king per row/column. To see this, we can make the first number of the permutation to represent the position of the king in the first row from the top left, the second number for the second row, and so on. We know that the positions cannot have a difference of $1$ because otherwise, two kings would attack each other. So, this problem is the same as finding the non-consecutive permutations.

\subsection{Small sizes}

For $n=1$, the only possible Latin square is non-consecutive. Non-consecutive Latin squares do not exist for $1 < n < 4$, because we can't construct a non-consecutive permutation. For $n=4$, such permutations exist, but there are only 2 of them, so not enough to form 4 rows of a square. Thus, the simplest case is $n=5$.

We look at shy Latin squares of size 5. The lexicographically earliest non-consecutive Latin square of order 5 is shown in Figure~\ref{fig:LENC5}.


\begin{figure}[ht!]
{
	\begin{center}
	\begin{tikzpicture}
	\draw[step=0.5cm,color=black, line width=1.5] (-1,-1) grid (1.5,1.5);
	\node at (-0.75, +1.25) {1};
	\node at (-0.25, +1.25) {3};
	\node at (+0.25, +1.25) {5};
	\node at (+0.75, +1.25) {2};
	\node at (+1.25, +1.25) {4};
	\node at (-0.75,+0.75) {3};
	\node at (-0.25,+0.75) {5};
	\node at (+0.25,+0.75) {2};
	\node at (+0.75,+0.75) {4};
	\node at (+1.25,+0.75) {1};
	\node at (-0.75,+0.25) {5};
	\node at (-0.25,+0.25) {2};
	\node at (+0.25,+0.25) {4};
	\node at (+0.75,+0.25) {1};
	\node at (+1.25,+0.25) {3};
	\node at (-0.75, -0.25) {2};
	\node at (-0.25, -0.25) {4};
	\node at (+0.25, -0.25) {1};
	\node at (+0.75, -0.25) {3};
	\node at (+1.25, -0.25) {5};
	\node at (-0.75,-0.75) {4};
	\node at (-0.25,-0.75) {1};
	\node at (+0.25,-0.75) {3};
	\node at (+0.75,-0.75) {5};
	\node at (+1.25,-0.75) {2};
	\end{tikzpicture}
	\end{center}
	\caption{The lexicographically earliest non-consecutive Latin square of order 5.}
	\label{fig:LENC5}
	}
\end{figure}

We see that it is a left-cyclic Latin square. We show that any non-consecutive Latin square of size 5 can be built in a similar way: by using a cyclable non-consecutive permutation and shifting it left or right by 1. We call this procedure \textit{cycling a cyclable permutation}.

\begin{proposition}\label{prop:5isCyclable}
Any non-consecutive Latin square of size 5 is built by cycling a cyclable permutation.
\end{proposition}

\begin{proof}
First, we prove that the first row must be a cyclable permutation, that is, it can't be one of 24153, 31524, 35142, and 42513. These four permutations are equivalent, in a sense that we can get to one from another by reversing the string or by swapping numbers 1 with 5 and 2 with 4. Thus, without loss of generality, we can assume that the first row is 24153. That means the next row must end in 1 or 5. We have only four non-consecutive permutations like that: 35241, 24135, 42531, and 31425. Only one of them can serve as the second row: 42531. We can see that we can't build the third row after this.

Thus, the first row is a cyclable permutation. By the same reasoning, any row is a cyclable permutation. In such permutations, 1 and 5 are never next to each other. Thus, every number in such a square has two options for neighbors. 

In another observation, we notice that as soon as we place two numbers in a row/column, the rest of the row/column is forced. 

Suppose we filled the first row, then every number has potentially two placements in the second row. It either shifts by 1 to the left or the right, but if the numbers shift in different directions, we have a contradiction. Thus the second row is either the left or the right shift of the first row. After we filled in the first two rows, the rest of the square is uniquely defined. We know that cycling works, so it has to be it.
\end{proof}

Now to toroidal shy Latin squares. We just described all possible non-consecutive Latin squares of order 5, and they all happen to be toroidal. They are also modular.

\begin{corollary}
All 5 by 5 non-consecutive Latin squares are also toroidal and modular non-consecutive Latin squares.
\end{corollary}

It makes us wonder whether non-toroidal (non-modular) shy squares exist. And they do, as can be seen in Figure~\ref{fig:ntnmnc}. 
\begin{figure}[ht!]
{
\begin{center}
\begin{tikzpicture}
\draw[step=0.5cm,color=black, line width=1.5] (-1,-1) grid (2,2);
\node at (-0.75, +1.75) {1};
\node at (-0.25, +1.75) {3};
\node at (+0.25, +1.75) {5};
\node at (+0.75, +1.75) {2};
\node at (+1.25, +1.75) {4};
\node at (+1.75, +1.75) {6};
\node at (-0.75, +1.25) {3};
\node at (-0.25, +1.25) {5};
\node at (+0.25, +1.25) {1};
\node at (+0.75, +1.25) {4};
\node at (+1.25, +1.25) {6};
\node at (+1.75, +1.25) {2};
\node at (-0.75,+0.75) {5};
\node at (-0.25,+0.75) {1};
\node at (+0.25,+0.75) {3};
\node at (+0.75,+0.75) {6};
\node at (+1.25,+0.75) {2};
\node at (+1.75, +0.75) {4};
\node at (-0.75,+0.25) {2};
\node at (-0.25,+0.25) {4};
\node at (+0.25,+0.25) {6};
\node at (+0.75,+0.25) {1};
\node at (+1.25,+0.25) {3};
\node at (+1.75, +0.25) {5};
\node at (-0.75, -0.25) {4};
\node at (-0.25, -0.25) {6};
\node at (+0.25, -0.25) {2};
\node at (+0.75, -0.25) {3};
\node at (+1.25, -0.25) {5};
\node at (+1.75, -0.25) {1};
\node at (-0.75,-0.75) {6};
\node at (-0.25,-0.75) {2};
\node at (+0.25,-0.75) {4};
\node at (+0.75,-0.75) {5};
\node at (+1.25,-0.75) {1};
\node at (+1.75, -0.75) {3};
\end{tikzpicture}
	\end{center}
	\caption{Not a toroidal or a modular non-consecutive Latin square.}
	\label{fig:ntnmnc}
	}
\end{figure}
This square is not a toroidal non-consecutive Latin square as, for example, 2 and 3 are at different ends of the second row. It is not a modular square as 1 and 6 are neighbors in the fourth row. It is interesting to note that our square can be divided into two pairs of smaller Latin squares, shown in Figure~\ref{fig:tss}, that form a checkerboard pattern.
\begin{figure}[ht!]
{
	\begin{center}
	\begin{tikzpicture}
	\draw[step=0.5cm,color=black, line width=1.5] (-1,-1) grid (0.5,0.5);
	\node at (-0.75,+0.25) {1};
	\node at (-0.25,+0.25) {3};
	\node at (+0.25,+0.25) {5};
	\node at (-0.75, -0.25) {3};
	\node at (-0.25, -0.25) {5};
	\node at (+0.25, -0.25) {1};
	\node at (-0.75,-0.75) {5};
	\node at (-0.25,-0.75) {1};
	\node at (+0.25,-0.75) {3};
	\end{tikzpicture}
	\quad \quad
	\begin{tikzpicture}
	\draw[step=0.5cm,color=black, line width=1.5] (-1,-1) grid (0.5,0.5);
	\node at (-0.75,+0.25) {2};
	\node at (-0.25,+0.25) {4};
	\node at (+0.25,+0.25) {6};
	\node at (-0.75, -0.25) {4};
	\node at (-0.25, -0.25) {6};
	\node at (+0.25, -0.25) {2};
	\node at (-0.75,-0.75) {6};
	\node at (-0.25,-0.75) {2};
	\node at (+0.25,-0.75) {4};
	\end{tikzpicture}
	\end{center}
	\caption{Two smaller Latin squares forming a checkerboard pattern.}
	\label{fig:tss}
	}
\end{figure}

\subsection{Large sizes}

Both toroidal and non-toroidal non-consecutive Latin squares of larger orders exist. It follows from the fact that the corresponding toroidal non-consecutive and non-toroidal non-consecutive permutations exist. After that, we can just cycle such permutations. The same is true for modular and non-modular Latin squares.

There are many other, not necessarily cyclic, such squares that exist. For example, consider a modular toroidal non-consecutive Latin square $N$ of order $n$ and any Latin square $M$ of order $m$. Consider the product of $M$ and $N$. We want to show that the result is a toroidal non-consecutive Latin square of order $mn$. The large square can be divided into $n$ by $n$ blocks, where each block equals to square $N$ incremented by a fixed number. That means, within a block, the numbers are not consecutive. 

Now we should look at boundaries between the blocks. It can fail to be non-consecutive only if the square $N$ has a row or column beginning with 1 and ending with $n$ or beginning with $n$ and ending with 1. By our assumption, the square $N$ is modular toroidal, so this can't happen.

\subsection{Non-consecutive by king's move}\label{sec:NonconbyKM}

We can also request that diagonal neighbors are non-consecutive integers. For order 1, such Latin squares trivially exist. For $1< n < 5$, such Latin squares do not exist because, as we showed in Section~\ref{sec:NC}, non-consecutive Latin squares do not exist for these sizes, so non-consecutive by king's move don't exist as well. 

Shy Latin squares of order 5 exist, but we showed that they are cyclic with a shift by 1 or $-1$. Hence, a number 3, which is not on the border, has to be adjacent to 1 and 5 in the same row, meaning it is diagonally adjacent to 2 or 4. Thus, these squares are not shy by king's move, so order 5 non-consecutive by king's move Latin squares don't exist.

What about size 6? Consider a row that starts with 1. The possibilities are 135246, 135264, 136425, 142536, 142635, 146253, 146352, 152463, 153624, 153642, 163524, and 164253. Notice that if a 2 is next to a 5, then the two numbers directly below those two can't be any of the other valid numbers because all of them are consecutive to 2 or 5. Thus, the numbers below have to be a 2 and a 5 in a different order. Then, we can't extend this square in any direction. Thus, the only possibilities that are left are 142635, 153624, and 153642. 

For 142635 and 153624, the number underneath the 6 can't be 6 and can't be consecutive to either of 2, 3, or 6. Thus, no number can be placed there, so this square can't be finished. For 153642, we can see that the only number that fits under 3 is 1. In addition, the only number that fits under 6 is 1. As we can't place 1 in two places at once, this square can't be finished. As a non-consecutive by king's move Latin square of size 6 must have a row that starts with a 1, we can see that it is impossible to create such a square.

One might find it odd to disallow diagonal neighbors to be consecutive and allow them to be the same. Latin squares that do not allow the same number to be a diagonal neighbor are called anti-king and are studied in Section~\ref{sec:anti-king}. So the natural question is to look for two types of Latin squares: non-consecutive by king's move and non-consecutive by king's move that is also anti-king.

From our discussion above, none of those squares exist for sizes 2, 3, 4, 5, and 6. 

\begin{theorem}
For all odd $n\ge 9$, there exists an $n$ by $n$ non-consecutive by king's move Latin square that is also anti-king. For $n=7$, a non-consecutive by king's move Latin square exists, but one that is also anti-king doesn't exist.
\end{theorem}

\begin{proof}
Given an odd $n \ge 9$, we build a non-consecutive by king's move Latin square. We start the first row with 1 and then increase by 2 for the next element, wrapping around after $n$. That means $n$ is followed by 2, then 4, and so on: $135\ldots$. Because our number is odd, the first row is a permutation of the numbers 1 through $n$. Two neighboring numbers differ by $2$, and, hence, the first row is non-consecutive. By symmetry, we can use the first row that cycles not by 2 but by $n-2$.

To make a non-consecutive by king's move Latin square, we 1-cycle the first row. For example, the element $a$ will have its neighbors, including diagonal neighbors, come from the following list: $a-4$, $a-2$, $a$, $a+2$, and $a+4$. For $n>5$, numbers $\pm 2$ and $\pm 4$ are not equal to $\pm 1$ modulo $n$. Thus the square is a non-consecutive by king's move Latin square.

To make a non-consecutive by king's move Latin square that is also anti-king, we can cycle the rows by $-2$, that is by 2 to the left. For example, the element $a$ will have its neighbors, including diagonal neighbors, come from the following list: $a-6$, $a-4$, $a-2$, $a+2$, $a+4$, and $a+6$. None of them are equal to $a$. Moreover, for $n>7$, numbers $\pm 2$, $\pm 4$, and $\pm 6$ are not equal to $\pm 1$ modulo $n$. Thus, the Latin square is a non-consecutive square by king's move that is also anti-king. 

Now we prove that a non-consecutive by king's move Latin square that is also anti-king doesn't exist for order 7. Consider the number 4 that is placed in such a square, not on the border. Its neighbors must be from the set $\{1,2,6,7\}$. The cells that are orthogonally adjacent to the cell with 4 have to be all distinct as they are either in the same row/column or are neighbors by king's move. This means 1 and 2 are opposite each other, and the same is true for 6 or 7. In this arrangement, we can't place any more numbers from the set $\{1,2,6,7\}$ into any of the remaining 4 spots.
\end{proof}

A 7 by 7 non-consecutive by king's move Latin square is shown on the left in Figure~\ref{fig:nck9}. It was created by cycling the first row to the left by 1. A 9 by 9 non-consecutive by king's move Latin square that is also anti-king is shown on the right in Figure~\ref{fig:nck9}. Its first row is built by adding 7 (which equals $n-2$) modulo 9 to the previous element. Thus, the square is a 2-cyclic square.
\begin{figure}[ht!]
	{
    \begin{center}
        \begin{tikzpicture}
            \draw[step=0.5cm,color=black, line width=1.5] (0,0) grid (3.5,3.5);
            \node at (0.25, 3.25) {1};
            \node at (0.75, 3.25) {3};
            \node at (1.25, 3.25) {5};
            \node at (1.75, 3.25) {7};
            \node at (2.25, 3.25) {2};
            \node at (2.75, 3.25) {4};
            \node at (3.25, 3.25) {6};
            
            \node at (0.25, 2.75) {3};
            \node at (0.75, 2.75) {5};
            \node at (1.25, 2.75) {7};
            \node at (1.75, 2.75) {2};
            \node at (2.25, 2.75) {4};
            \node at (2.75, 2.75) {6};
            \node at (3.25, 2.75) {1};
            
            \node at (0.25, 2.25) {5};
            \node at (0.75, 2.25) {7};
            \node at (1.25, 2.25) {2};
            \node at (1.75, 2.25) {4};
            \node at (2.25, 2.25) {6};
            \node at (2.75, 2.25) {1};
            \node at (3.25, 2.25) {3};
            
            \node at (0.25, 1.75) {7};
            \node at (0.75, 1.75) {2};
            \node at (1.25, 1.75) {4};
            \node at (1.75, 1.75) {6};
            \node at (2.25, 1.75) {1};
            \node at (2.75, 1.75) {3};
            \node at (3.25, 1.75) {5};
            
            \node at (0.25, 1.25) {2};
            \node at (0.75, 1.25) {4};
            \node at (1.25, 1.25) {6};
            \node at (1.75, 1.25) {1};
            \node at (2.25, 1.25) {3};
            \node at (2.75, 1.25) {5};
            \node at (3.25, 1.25) {7};
            
            \node at (0.25, 0.75) {4};
            \node at (0.75, 0.75) {6};
            \node at (1.25, 0.75) {1};
            \node at (1.75, 0.75) {3};
            \node at (2.25, 0.75) {5};
            \node at (2.75, 0.75) {7};
            \node at (3.25, 0.75) {2};
            
            \node at (0.25, 0.25) {6};
            \node at (0.75, 0.25) {1};
            \node at (1.25, 0.25) {3};
            \node at (1.75, 0.25) {5};
            \node at (2.25, 0.25) {7};
            \node at (2.75, 0.25) {2};
            \node at (3.25, 0.25) {4};
        \end{tikzpicture}
\quad \quad 
	\begin{tikzpicture}
	\draw[step=0.5cm,color=black, line width=1.5] (-2.5,-2.5) grid (2,2);
	\node at (-2.25, +1.75){1};
	\node at (-2.25, +1.25){5};
	\node at (-2.25, +0.75){9};
	\node at (-2.25, +0.25){4};
	\node at (-2.25, -0.25){8};
	\node at (-2.25, -0.75){3};
	\node at (-2.25, -1.25){7};
	\node at (-2.25, -1.75){2};
	\node at (-2.25, -2.25){6};
	\node at (-1.75, +1.75){8};
	\node at (-1.75, +1.25){3};
	\node at (-1.75, +0.75){7};
	\node at (-1.75, +0.25){2};
	\node at (-1.75, -0.25){6};
	\node at (-1.75, -0.75){1};
	\node at (-1.75, -1.25){5};
	\node at (-1.75, -1.75){9};
	\node at (-1.75, -2.25){4};
	\node at (-1.25, +1.75){6};
	\node at (-1.25, +1.25){1};
	\node at (-1.25, +0.75){5};
	\node at (-1.25, +0.25){9};
	\node at (-1.25, -0.25){4};
	\node at (-1.25, -0.75){8};
	\node at (-1.25, -1.25){3};
	\node at (-1.25, -1.75){7};
	\node at (-1.25, -2.25){2};
	\node at (-0.75, +1.75){4};
	\node at (-0.75, +1.25){8};
	\node at (-0.75, +0.75){3};
	\node at (-0.75, +0.25){7};
	\node at (-0.75, -0.25){2};
	\node at (-0.75, -0.75){6};
	\node at (-0.75, -1.25){1};
	\node at (-0.75, -1.75){5};
	\node at (-0.75, -2.25){9};
	\node at (-0.25, +1.75){2};
	\node at (-0.25, +1.25){6};
	\node at (-0.25, +0.75){1};
	\node at (-0.25, +0.25){5};
	\node at (-0.25, -0.25){9};
	\node at (-0.25, -0.75){4};
	\node at (-0.25, -1.25){8};
	\node at (-0.25, -1.75){3};
	\node at (-0.25, -2.25){7};
	\node at (0.25, +1.75){9};
	\node at (0.25, +1.25){4};
	\node at (0.25, +0.75){8};
	\node at (0.25, +0.25){3};
	\node at (0.25, -0.25){7};
	\node at (0.25, -0.75){2};
	\node at (0.25, -1.25){6};
	\node at (0.25, -1.75){1};
	\node at (0.25, -2.25){5};
	\node at (0.75, +1.75){7};
	\node at (0.75, +1.25){2};
	\node at (0.75, +0.75){6};
	\node at (0.75, +0.25){1};
	\node at (0.75, -0.25){5};
	\node at (0.75, -0.75){9};
	\node at (0.75, -1.25){4};
	\node at (0.75, -1.75){8};
	\node at (0.75, -2.25){3};
	\node at (1.25, +1.75){5};
	\node at (1.25, +1.25){9};
	\node at (1.25, +0.75){4};
	\node at (1.25, +0.25){8};
	\node at (1.25, -0.25){3};
	\node at (1.25, -0.75){7};
	\node at (1.25, -1.25){2};
	\node at (1.25, -1.75){6};
	\node at (1.25, -2.25){1};
	\node at (1.75, +1.75){3};
	\node at (1.75, +1.25){7};
	\node at (1.75, +0.75){2};
	\node at (1.75, +0.25){6};
	\node at (1.75, -0.25){1};
	\node at (1.75, -0.75){5};
	\node at (1.75, -1.25){9};
	\node at (1.75, -1.75){4};
	\node at (1.75, -2.25){8};
	\end{tikzpicture}
	\end{center}
	\caption{Non-consecutive by king's move Latin squares of orders 7 and 9. The one of order 9 is also anti-king.}
	\label{fig:nck9}
	}
\end{figure}

We have discussed odd orders. What about even orders? Figure~\ref{fig:AncbykmLsnak} shows an order 8 non-consecutive by king's move Latin square, which is not an anti-king Latin square. This non-consecutive by king's move Latin square is made by right-cycling a row in which a number does not have a consecutive number within a cell toroidally within 2 cells of itself. Such row can be described as follows: the next number equals the previous number plus 3 modulo 8.
\begin{figure}[ht!]
{
	\begin{center}
	\begin{tikzpicture}
	\draw[step=0.5cm,color=black, line width=1.5] (0,0) grid (4,4);
	\node at (0.25,3.75) {1};
	\node at (0.75,3.75) {4};
	\node at (1.25,3.75) {7};
	\node at (1.75,3.75) {2};
	\node at (2.25,3.75) {5};
	\node at (2.75,3.75) {8};
	\node at (3.25,3.75) {3};
	\node at (3.75,3.75) {6};
	
	\node at (0.25,3.25) {6};
	\node at (0.75,3.25) {1};
	\node at (1.25,3.25) {4};
	\node at (1.75,3.25) {7};
	\node at (2.25,3.25) {2};
	\node at (2.75,3.25) {5};
	\node at (3.25,3.25) {8};
	\node at (3.75,3.25) {3};
	
	\node at (0.25,2.75) {3};
	\node at (0.75,2.75) {6};
	\node at (1.25,2.75) {1};
	\node at (1.75,2.75) {4};
	\node at (2.25,2.75) {7};
	\node at (2.75,2.75) {2};
	\node at (3.25,2.75) {5};
	\node at (3.75,2.75) {8};
	
	\node at (0.25,2.25) {8};
	\node at (0.75,2.25) {3};
	\node at (1.25,2.25) {6};
	\node at (1.75,2.25) {1};
	\node at (2.25,2.25) {4};
	\node at (2.75,2.25) {7};
	\node at (3.25,2.25) {2};
	\node at (3.75,2.25) {5};
	
	\node at (0.25,1.75) {5};
	\node at (0.75,1.75) {8};
	\node at (1.25,1.75) {3};
	\node at (1.75,1.75) {6};
	\node at (2.25,1.75) {1};
	\node at (2.75,1.75) {4};
	\node at (3.25,1.75) {7};
	\node at (3.75,1.75) {2};
	
	\node at (0.25,1.25) {2};
	\node at (0.75,1.25) {5};
	\node at (1.25,1.25) {8};
	\node at (1.75,1.25) {3};
	\node at (2.25,1.25) {6};
	\node at (2.75,1.25) {1};
	\node at (3.25,1.25) {4};
	\node at (3.75,1.25) {7};
	
	\node at (0.25,0.75) {7};
	\node at (0.75,0.75) {2};
	\node at (1.25,0.75) {5};
	\node at (1.75,0.75) {8};
	\node at (2.25,0.75) {3};
	\node at (2.75,0.75) {6};
	\node at (3.25,0.75) {1};
	\node at (3.75,0.75) {4};
	
	\node at (0.25,0.25) {4};
	\node at (0.75,0.25) {7};
	\node at (1.25,0.25) {2};
	\node at (1.75,0.25) {5};
	\node at (2.25,0.25) {8};
	\node at (2.75,0.25) {3};
	\node at (3.25,0.25) {6};
	\node at (3.75,0.25) {1};
	\end{tikzpicture}
	\end{center}
	\caption{A non-consecutive by king's move Latin square that is not anti-king.}
	\label{fig:AncbykmLsnak}
}
\end{figure}

Now we discuss the existence of shy by king's move Latin squares and shy by king's move Latin squares that are also anti-king exist for even orders. We prove the existence by using a construction similar to the one used to build the square in Figure~\ref{fig:AncbykmLsnak}.

\begin{theorem}
For an even $n$ greater than 6, there exists an $n$ by $n$ non-consecutive by king's move Latin square. For an even $n$ greater than 12, there exists an $n$ by $n$ non-consecutive by king's move Latin square that is also anti-king.
\end{theorem}

\begin{proof}
We start with a more restrictive construction of a non-consecutive Latin square that is also anti-king. We build such a Latin square, where the first row is made by successively adding $k$ modulo $n$ with a shift of $m$ between rows. For it to be a Latin square, the numbers $k$ and $m$ must be coprime with $n$. 

The neighborhood of any number $a$ in such a square is as follows:
\[
\begin{matrix}
a-k-mk & a-mk & a+k-mk \\
a-k & a & k \\
a-k+mk & a+mk & a+k+mk.
\end{matrix}
\]

For this square to be anti-king we need $mk\pm k \neq 0$ modulo $n$, which is equivalent to $m \neq \pm 1$ modulo $n$, since $k$ is coprime with $n$. For this square to be non-consecutive by king's move, we need that $k$, $mk$, and $mk\pm k$ not equal $\pm 1$ modulo $n$. The numbers $mk\pm k$ are even so they can't be equal $\pm 1$ modulo $n$, since $n$ is also even.

This means the final set of conditions that is enough for our construction is that $k$ and $m$ are coprime with $n$ and $k$, $m$ and $mk$ are not equal $\pm 1$ modulo $n$.

Now fix $m$ satisfying these constraints. There are $\phi(n)$ possibilities for $k$ to be coprime with $n$. The condition $k\neq\pm 1$ eliminates two of them. In addition, since $m$ is coprime with $n$, there is a unique $k$ for which $mk=1$ and also a unique $k$ for which $mk=-1$. We can eliminate these two cases as well. So the total number of choices for $k$ is at least $\phi(n)-4$. If $\phi(n)\ge 5$, then we can always choose a $k$.

The fact that for even $n > 12$, the Euler phi function exceeds 5 concludes the proof of the second part.

For the first part we notice that we can choose $m = \pm 1$ modulo $n$. After that the requirement for $k$ is that it is coprime with $n$ and $k \neq \pm 1$ modulo $n$. Thus, we can always find such a $k$ for $n > 6$, as for such even number the Euler phi functions exceeds 2.
\end{proof}

Combining the two above theorems, we see that a non-consecutive by king's move Latin square that is also anti-king exists for any $n$ greater than 12.

\section{Consecutive Latin squares}\label{sec:Consecutive}

A \textit{consecutive Latin square} is a Latin square where two orthogonally adjacent cells must contain consecutive numbers. We also call them \textit{nosy}. As opposed to shy squares, numbers in nosy squares want to keep their neighbors close.

\subsection{Consecutive Latin squares}

All Latin squares of orders 1 and 2 are consecutive.

\begin{theorem}
Consecutive Latin squares do not exist for order $n$, where $n>2$.
\end{theorem}

\begin{proof}
There are only two possible ways to create a consecutive permutation of numbers 1 to $n$: to put them in increasing or decreasing order. A consecutive Latin square is formed by stacking complete consecutive rows on top of each other. But with only two possible rows to stack, the order of the square must be either 1 or 2, and no greater.
\end{proof}

Toroidal Latin squares are more restrictive, so toroidal consecutive Latin squares of order $n$, where $n > 2$, do not exist either.

\subsection{Modular consecutive Latin squares}

A \textit{modular consecutive Latin square} is a Latin square where two orthogonally adjacent cells must contain consecutive numbers, similar to a consecutive Latin square. However, unlike in a consecutive Latin square of order $n$, in a modular consecutive Latin square of order $n > 2$, numbers 1 and $n$ are considered consecutive.

Unlike consecutive Latin squares, modular consecutive Latin squares exist for any $n$. For example, a modular consecutive Latin square of size $n$ can be created by cycling the numbers in increasing (or decreasing) order from 1 to $n$ with a shift of 1 or $-1$.

\begin{proposition}
A modular consecutive Latin square of order $n > 4$ is either left- or right-cyclic.
\end{proposition}

\begin{proof}
In a modular nosy Latin square, each number has two possibilities for neighboring numbers. This means that for each row and column, we only need to fill in two adjacent numbers, after which all placements are forced: the numbers will go in an increasing or decreasing cycle. 

Without loss of generality, we assume that the first row is in increasing order. There are two options for the left-most cell of the second row: 2 or $n$. 

If we choose 2, then the next cell to the right can't be a 1, as this would force the third number in this row to be consecutive with both 1 and 3, and it cannot be 2 (for $n = 4$ the number can be 4, which creates an interesting exception). This means that the next cell to the right must be a 3. From here, the rest of the row is forced, meaning the rest of the columns are forced, yielding a left-cyclic Latin square.

If we instead choose $n$, then the second cell in the second row must be consecutive to both 2 and $n$. If $n > 4$, then there is only one option: 1. From here, the rest of the row is forced, meaning the rest of the columns are forced, yielding a right-cyclic Latin square.

Thus, all modular consecutive Latin squares for orders $n > 4$ are left- or right-cyclic.
\end{proof}

In an interesting corollary, we see that a modular nosy Latin square is uniquely determined by its three top-left entries.

\begin{corollary}
A modular consecutive Latin square of order more than 4 is uniquely determined by its entries with coordinates $(1,1)$, $(1,2)$, and $(2,1)$.
\end{corollary}

\begin{proof}
The entries with coordinates $(1,1)$ and $(1,2)$ uniquely define the first row. The entry with coordinates $(2,1)$ defines the left or right shift.
\end{proof}

This corollary produces its own corollary.

\begin{corollary}
For $n > 4$ there are $4n$ modular consecutive Latin squares.
\end{corollary}

\begin{proof}
There are $n$ ways to choose the entry with coordinates $(1,1)$. After that there are 2 ways to choose each of the entries $(1,2)$ and $(2,1)$.
\end{proof}

Interestingly, it is possible to make a non-cyclic modular consecutive Latin square of order 4. Similarly to our discussion above, we can show that such a square is uniquely defined by its top-left 2 by 2 square. Without loss of generality, let the first row of this small square be 12. Then the second row could be 21, 23, 41, and 43. Rows 23 and 41 correspond to cyclic squares. Therefore, there are exactly two different modular consecutive Latin squares of order 4 with the first row being in order that are not cyclic. These squares are shown in Figure~\ref{fig:NCMC4}.
\begin{figure}[ht!]
{
	\begin{center}
	\begin{tikzpicture}
	\draw[step=0.5cm,color=black, line width=1.5] (-1,-1) grid (1,1);
	\node at (-0.75,+0.75) {1};
	\node at (-0.25,+0.75) {2};
	\node at (+0.25,+0.75) {3};
	\node at (+0.75,+0.75) {4};
	\node at (-0.75,+0.25) {4};
	\node at (-0.25,+0.25) {3};
	\node at (+0.25,+0.25) {2};
	\node at (+0.75,+0.25) {1};
	\node at (-0.75, -0.25) {3};
	\node at (-0.25, -0.25) {4};
	\node at (+0.25, -0.25) {1};
	\node at (+0.75, -0.25) {2};
	\node at (-0.75,-0.75) {2};
	\node at (-0.25,-0.75) {1};
	\node at (+0.25,-0.75) {4};
	\node at (+0.75,-0.75) {3};
	\end{tikzpicture}
\quad \quad 
	\begin{tikzpicture}
	\draw[step=0.5cm,color=black, line width=1.5] (-1,-1) grid (1,1);
	\node at (-0.75,+0.75) {1};
	\node at (-0.25,+0.75) {2};
	\node at (+0.25,+0.75) {3};
	\node at (+0.75,+0.75) {4};
	\node at (-0.75,+0.25) {2};
	\node at (-0.25,+0.25) {1};
	\node at (+0.25,+0.25) {4};
	\node at (+0.75,+0.25) {3};
	\node at (-0.75, -0.25) {3};
	\node at (-0.25, -0.25) {4};
	\node at (+0.25, -0.25) {1};
	\node at (+0.75, -0.25) {2};
	\node at (-0.75,-0.75) {4};
	\node at (-0.25,-0.75) {3};
	\node at (+0.25,-0.75) {2};
	\node at (+0.75,-0.75) {1};
	\end{tikzpicture}
	\end{center}
	\caption{Non-cyclic modular consecutive Latin squares of order 4.}
	\label{fig:NCMC4}
	}
\end{figure}

All consecutive modular Latin squares of order more than 4 are toroidal because they are cyclic. These two exceptional examples for $n=4$ are also toroidal. Thus, all consecutive modular Latin squares are also toroidal.

\subsection{Special number 5}

Consider a modular toroidal case. For any number, two numbers are allowed to be its neighbors. Now we want to reminisce about the non-consecutive case when $n=5$. In this case, each number again has precisely two numbers that are allowed to be neighbors. Therefore, we can put them in cycle 13524. Thus, we can consider a bijection on numbers that turns row 12345 into 13524. This bijection between numbers creates a bijection between modular toroidal shy Latin squares and modular toroidal nosy Latin squares of order 5.

\section{Anti-chess Latin squares}\label{sec:AC}

We want to consider an unspecified chess piece. We give it a special name \textit{chiece}. This word is a portmanteau word combining the two words chess and piece. This section was inspired by popular Sudoku variation known as anti-knight Sudoku.

An \textit{anti-chiece} Latin square is a Latin square such that any two cells separated by a chiece's move can't be the same. For example, anti-rook Latin squares are just Latin squares. Also, anti-queen and anti-bishop Latin squares are the same.

If we have an anti-chiece Latin square, then we can relabel the numbers in any order and still have such a square. Thus, without loss of generality, we can assume that the first row is in order. Trivially, a Latin square of order 1 is an anti-chiece Latin square.

\subsection{Anti-knight Latin squares}

\textit{Anti-knight} Latin squares are Latin squares such that any two numbers separated by a knight's move are distinct.

Latin squares of sizes 1 and 2 have to be anti-knight, as no two cells in such squares are a knight's move apart. Size-3 anti-knight Latin squares do not exist. The good news: we already saw an anti-knight Latin square of order 4 in Figure~\ref{fig:LE4}: the lexicographically earliest Latin square of order 4 is also an anti-knight Latin square.

Another example of an anti-knight Latin square is shown in Figure~\ref{fig:Sudoku}. It forms a Sudoku grid, which means that it can be divided into four 2 by 2 boxes, each containing distinct numbers. We also see that it is formed by placing two identical 2 by 2 Latin squares in a checkerboard pattern.
\begin{figure}[ht!]
	{
	\begin{center}
	\begin{tikzpicture}
	\draw[step=0.5cm,color=black, line width=1.5] (-1,-1) grid (1,1);
	\node at (-0.75,+0.75) {1};
	\node at (-0.25,+0.75) {2};
	\node at (+0.25,+0.75) {3};
	\node at (+0.75,+0.75) {4};
	\node at (-0.75,+0.25) {4};
	\node at (-0.25,+0.25) {3};
	\node at (+0.25,+0.25) {2};
	\node at (+0.75,+0.25) {1};
	\node at (-0.75, -0.25) {3};
	\node at (-0.25, -0.25) {4};
	\node at (+0.25, -0.25) {1};
	\node at (+0.75, -0.25) {2};
	\node at (-0.75,-0.75) {2};
	\node at (-0.25,-0.75) {1};
	\node at (+0.25,-0.75) {4};
	\node at (+0.75,-0.75) {3};
	\end{tikzpicture}
	\end{center}
	\caption{An anti-knight Latin square that forms a Sudoku.}
	\label{fig:Sudoku}
	}
\end{figure}

One might ask whether all Latin squares of size 4 are anti-knight. The answer is no: Figure~\ref{fig:NAK} shows an example.
\begin{figure}[ht!]
	{
	\begin{center}
	\begin{tikzpicture}
	\draw[step=0.5cm,color=black, line width=1.5] (-1,-1) grid (1,1);
	\node at (-0.75,+0.75) {1};
	\node at (-0.25,+0.75) {2};
	\node at (+0.25,+0.75) {3};
	\node at (+0.75,+0.75) {4};
	\node at (-0.75,+0.25) {3};
	\node at (-0.25,+0.25) {4};
	\node at (+0.25,+0.25) {1};
	\node at (+0.75,+0.25) {2};
	\node at (-0.75, -0.25) {2};
	\node at (-0.25, -0.25) {3};
	\node at (+0.25, -0.25) {4};
	\node at (+0.75, -0.25) {1};
	\node at (-0.75,-0.75) {4};
	\node at (-0.25,-0.75) {1};
	\node at (+0.25,-0.75) {2};
	\node at (+0.75,-0.75) {3};
	\end{tikzpicture}
	\end{center}
	\caption{Not an anti-knight Latin square.}
	\label{fig:NAK}
	}
\end{figure}

We want to find how many anti-knight Latin squares of order 4 exist. All possible anti-knight Latin squares with the first row in order are in Figure~\ref{fig:AllAK}. 
\begin{figure}[ht!]
{
	\begin{center}
	\begin{tikzpicture}
	\draw[step=0.5cm,color=black, line width=1.5] (-1,-1) grid (1,1);
	\node at (-0.75,+0.75) {1};
	\node at (-0.25,+0.75) {2};
	\node at (+0.25,+0.75) {3};
	\node at (+0.75,+0.75) {4};
	\node at (-0.75,+0.25) {2};
	\node at (-0.25,+0.25) {1};
	\node at (+0.25,+0.25) {4};
	\node at (+0.75,+0.25) {3};
	\node at (-0.75, -0.25) {3};
	\node at (-0.25, -0.25) {4};
	\node at (+0.25, -0.25) {1};
	\node at (+0.75, -0.25) {2};
	\node at (-0.75,-0.75) {4};
	\node at (-0.25,-0.75) {3};
	\node at (+0.25,-0.75) {2};
	\node at (+0.75,-0.75) {1};
	\end{tikzpicture}
\quad \quad
	\begin{tikzpicture}
	\draw[step=0.5cm,color=black, line width=1.5] (-1,-1) grid (1,1);
	\node at (-0.75,+0.75) {1};
	\node at (-0.25,+0.75) {2};
	\node at (+0.25,+0.75) {3};
	\node at (+0.75,+0.75) {4};
	\node at (-0.75,+0.25) {4};
	\node at (-0.25,+0.25) {1};
	\node at (+0.25,+0.25) {2};
	\node at (+0.75,+0.25) {3};
	\node at (-0.75, -0.25) {3};
	\node at (-0.25, -0.25) {4};
	\node at (+0.25, -0.25) {1};
	\node at (+0.75, -0.25) {2};
	\node at (-0.75,-0.75) {2};
	\node at (-0.25,-0.75) {3};
	\node at (+0.25,-0.75) {4};
	\node at (+0.75,-0.75) {1};
	\end{tikzpicture}
\quad \quad
	\begin{tikzpicture}
	\draw[step=0.5cm,color=black, line width=1.5] (-1,-1) grid (1,1);
	\node at (-0.75,+0.75) {1};
	\node at (-0.25,+0.75) {2};
	\node at (+0.25,+0.75) {3};
	\node at (+0.75,+0.75) {4};
	\node at (-0.75,+0.25) {2};
	\node at (-0.25,+0.25) {3};
	\node at (+0.25,+0.25) {4};
	\node at (+0.75,+0.25) {1};
	\node at (-0.75, -0.25) {3};
	\node at (-0.25, -0.25) {4};
	\node at (+0.25, -0.25) {1};
	\node at (+0.75, -0.25) {2};
	\node at (-0.75,-0.75) {4};
	\node at (-0.25,-0.75) {1};
	\node at (+0.25,-0.75) {2};
	\node at (+0.75,-0.75) {3};
	\end{tikzpicture}
\quad \quad
	\begin{tikzpicture}
	\draw[step=0.5cm,color=black, line width=1.5] (-1,-1) grid (1,1);
	\node at (-0.75,+0.75) {1};
	\node at (-0.25,+0.75) {2};
	\node at (+0.25,+0.75) {3};
	\node at (+0.75,+0.75) {4};
	\node at (-0.75,+0.25) {4};
	\node at (-0.25,+0.25) {3};
	\node at (+0.25,+0.25) {2};
	\node at (+0.75,+0.25) {1};
	\node at (-0.75, -0.25) {3};
	\node at (-0.25, -0.25) {4};
	\node at (+0.25, -0.25) {1};
	\node at (+0.75, -0.25) {2};
	\node at (-0.75,-0.75) {2};
	\node at (-0.25,-0.75) {1};
	\node at (+0.25,-0.75) {4};
	\node at (+0.75,-0.75) {3};
	\end{tikzpicture}
	\end{center}
	\caption{All anti-knight Latin squares of order 4 up to relabeling.}
	\label{fig:AllAK}
	}
\end{figure}

There are four Latin squares. We can notice that the diagonal is either the same number or two pairs of the same numbers. We can also notice that in all of the above squares, the third row is the same. This is not a coincidence. In an anti-knight Latin square of order 4, the third row is uniquely defined after the first row is given. We leave it to the readers to explain it. It follows that there are 96 anti-knight Latin squares of order 4.

Interestingly, in size 5, there are fewer anti-knight Latin squares with a fixed first row: there are only two of them. The first example is in Figure~\ref{fig:AK5}.
\begin{figure}[ht!]
	{
	\begin{center}
	\begin{tikzpicture}
	\draw[step=0.5cm,color=black, line width=1.5] (-1,-1) grid (1.5,1.5);
	\node at (-0.75, +1.25) {1};
	\node at (-0.25, +1.25) {2};
	\node at (+0.25, +1.25) {3};
	\node at (+0.75, +1.25) {4};
	\node at (+1.25, +1.25) {5};
	\node at (-0.75,+0.75) {2};
	\node at (-0.25,+0.75) {3};
	\node at (+0.25,+0.75) {4};
	\node at (+0.75,+0.75) {5};
	\node at (+1.25,+0.75) {1};
	\node at (-0.75,+0.25) {3};
	\node at (-0.25,+0.25) {4};
	\node at (+0.25,+0.25) {5};
	\node at (+0.75,+0.25) {1};
	\node at (+1.25,+0.25) {2};
	\node at (-0.75, -0.25) {4};
	\node at (-0.25, -0.25) {5};
	\node at (+0.25, -0.25) {1};
	\node at (+0.75, -0.25) {2};
	\node at (+1.25, -0.25) {3};
	\node at (-0.75,-0.75) {5};
	\node at (-0.25,-0.75) {1};
	\node at (+0.25,-0.75) {2};
	\node at (+0.75,-0.75) {3};
	\node at (+1.25,-0.75) {4};
	\end{tikzpicture}
	\end{center}
	\caption{Anti-knight Latin square of order 5.}
	\label{fig:AK5}
	}
\end{figure}

In it, the second row is the cycling of the first row by 1 to the left. Similarly, we can cycle to the right. And this is it. It follows that any 5 by 5 Latin square can be built by starting with any permutation and constructing a left-cyclic or a right-cyclic square out of it. It follows that there 240 anti-knight Latin squares of order 5.

As we know all anti-knight Latin squares of order 5, we can know that the following proposition is true. 

\begin{proposition}
All 5 by 5 non-consecutive Latin squares are also anti-knight Latin squares.
\end{proposition}

Now to any order.

\begin{theorem}
Anti-knight Latin squares exist for any order $n > 3$.
\end{theorem}

\begin{proof}
We prove this by construction. Pick any sequence of numbers for the first row. Then build a cyclic square by shifting the first row by 1. Thus, the same number in the next row is diagonally adjacent. Assuming the square is colored in a checkerboard manner, a knight moves from one color to another, while cells on the same diagonals are of the same color. It follows that this square is anti-knight.
\end{proof}

The anti-knight Latin squares of size 5 we constructed are also toroidal anti-knight Latin squares. This connection works for any cyclic square that is shifted by 1. That means toroidal anti-knight Latin squares of any order do exist.

\subsection{Anti-queen Latin squares}

Since the traditional Latin square rules (each row and column must have one of each number) apply, anti-bishop Latin squares and anti-queen Latin squares are equivalent. In this paper, we prefer the term \textit{anti-queen}.

It is easy to check that anti-queen Latin squares do not exist for $1 < n < 4$. For $n = 4$, let us assume that the first row is 1234. Then there is only one way to place the second row as 3412. Then we can't fill in the third row. 

The question of finding such squares is related to the famous question of non-attacking queens. Indeed, all entries corresponding to the same number $k$ in an anti-queen Latin square are positions of $k$ non-attacking queens.

The latter question on the regular 8 by 8 chessboard is known as the eight queens puzzle, which is the problem of placing eight chess queens on an 8 by 8 chessboard so that no two queens can attack each other. It is known that there are 92 different solutions to this puzzle. Up to geometric transformations, there are 12.

There is a sequence in the OEIS devoted to non-attacking queens: A000170: Number of ways of placing $n$ non-attacking queens on an $n$ by $n$ board. 1, 1, 0, 0, 2, 10, 4, 40, 92, 352, 724, 2680, 14200, 73712, and so on. The sequence starts from the zero-sized board. 

For $n=5$, we have 10 ways to place non-attacking queens. Is it possible to combine 5 of them into one square? Yes, it is possible. If our first row is in order, there are two anti-queen Latin squares: they are shown in Figure~\ref{fig:AQ5}.

\begin{figure}[ht!]
{
	\begin{center}
	\begin{tikzpicture}
	\draw[step=0.5cm,color=black, line width=1.5] (-1,-1) grid (1.5,1.5);
	\node at (-0.75, +1.25) {1};
	\node at (-0.25, +1.25) {2};
	\node at (+0.25, +1.25) {3};
	\node at (+0.75, +1.25) {4};
	\node at (+1.25, +1.25) {5};
	\node at (-0.75,+0.75) {3};
	\node at (-0.25,+0.75) {4};
	\node at (+0.25,+0.75) {5};
	\node at (+0.75,+0.75) {1};
	\node at (+1.25,+0.75) {2};
	\node at (-0.75,+0.25) {5};
	\node at (-0.25,+0.25) {1};
	\node at (+0.25,+0.25) {2};
	\node at (+0.75,+0.25) {3};
	\node at (+1.25,+0.25) {4};
	\node at (-0.75, -0.25) {2};
	\node at (-0.25, -0.25) {3};
	\node at (+0.25, -0.25) {4};
	\node at (+0.75, -0.25) {5};
	\node at (+1.25, -0.25) {1};
	\node at (-0.75,-0.75) {4};
	\node at (-0.25,-0.75) {5};
	\node at (+0.25,-0.75) {1};
	\node at (+0.75,-0.75) {2};
	\node at (+1.25,-0.75) {3};
	\end{tikzpicture}
\quad \quad
	\begin{tikzpicture}
	\draw[step=0.5cm,color=black, line width=1.5] (-1,-1) grid (1.5,1.5);
	\node at (-0.75, +1.25) {1};
	\node at (-0.25, +1.25) {2};
	\node at (+0.25, +1.25) {3};
	\node at (+0.75, +1.25) {4};
	\node at (+1.25, +1.25) {5};
	\node at (-0.75,+0.75) {4};
	\node at (-0.25,+0.75) {5};
	\node at (+0.25,+0.75) {1};
	\node at (+0.75,+0.75) {2};
	\node at (+1.25,+0.75) {3};
	\node at (-0.75,+0.25) {2};
	\node at (-0.25,+0.25) {3};
	\node at (+0.25,+0.25) {4};
	\node at (+0.75,+0.25) {5};
	\node at (+1.25,+0.25) {1};
	\node at (-0.75, -0.25) {5};
	\node at (-0.25, -0.25) {1};
	\node at (+0.25, -0.25) {2};
	\node at (+0.75, -0.25) {3};
	\node at (+1.25, -0.25) {4};
	\node at (-0.75,-0.75) {3};
	\node at (-0.25,-0.75) {4};
	\node at (+0.25,-0.75) {5};
	\node at (+0.75,-0.75) {1};
	\node at (+1.25,-0.75) {2};
	\end{tikzpicture}
	\end{center}
	\caption{Anti-queen Latin squares of order 5.}
	\label{fig:AQ5}
	}
\end{figure}

Both of them are formed by cycling the previous row. For the first square, the shift is 2 to the left, and for the second, it is 2 to the right. Both of them are also toroidal anti-queen Latin squares.

The set of anti-queen Latin squares is a subset of the anti-king Latin squares, as king's moves are a subset of queen's moves.

Looking at larger $n$, we want to study the toroidal case as it is more restrictive. We have the following theorem.

\begin{theorem}
Shifting a permutation of length $n$ by $k$ in either direction produces a toroidal anti-queen Latin square of size $n$ if and only if $n$ is coprime with $k$, $k-1$, and $k+1$.
\end{theorem}

\begin{proof}
Suppose the first row is fixed. Suppose also that we number cells in a row starting with 0. Consider the number in the $j$-th cell of the first row. Shifting by $k$ puts this number in place $j +km\mod n$ in row $m$. 

For the resulting square to be a Latin square it is necessary and sufficient that $j +km_1 \neq j + km_2 \mod n$ for any pair of rows $m_1$ and $m_2$. Equivalently, $km_1 \neq km_2 \mod n$. Hence, the shifting method creates a Latin square from the first row if and only if $n$ is coprime with $k$.

The queen in the cell $j +km_1$ of row $m_1$ attacks pieces in row $m_2$ that are in the positions $j +km_1 \pm (m_2-m_1)\mod n$. Thus, we need $j +km_1 \pm (m_2-m_1) \neq j + km_2 \mod n$. This is equivalent to $(k \pm 1)(m2-m1) \neq 0 \mod n$ for any pair $m_1$ and $m_2$. This is equivalent to $(k \pm 1)$ being coprime with $n$.
\end{proof}

\begin{corollary}
A toroidal anti-queen Latin square produced by shifting method exists if and only if $n$ is coprime with 6.
\end{corollary}

\begin{proof}
Out of numbers $k-1$, $k$, and $k+1$, at least one is divisible by 2, and at least one is divisible by 3. Thus, if $n$ is not coprime with 6, then the shifting method doesn't produce an anti-queen Latin square. If $n$ is coprime with 6, then numbers 1, 2, 3 are coprime with $n$, and shifting by 2 produces the desired square.
\end{proof}

\subsection{Anti-king Latin squares}\label{sec:anti-king}

The trivial Latin square of order 1 is an anti-king Latin square. For size 2, an anti-king Latin square doesn't exist. We can't have an anti-king Latin square in size 3 either, as the number in the center forbids all other placements of the same number. For size 4, suppose the first row is 1234, then the second row is uniquely defined as 3412. After that, we can't place 1 and 4 anywhere in the third row. So we can't have it in size 4.

For larger sizes, the anti-king Latin squares exist as the anti-queen Latin squares are also anti-king Latin squares, and the latter squares exist. 

We can expect that there are more anti-king Latin squares than anti-queen Latin squares, but in size 5, these two sets are the same. Thus, the only two anti-king Latin squares of order 5 with the first row 12345 are shown in Figure~\ref{fig:AQ5}. The two squares can be transformed one into another through relabeling 1 as 5, 2 as 4, 4 as 2, and 5 as 1 and reflecting the square with respect to the middle column.

Figure~\ref{fig:akNotaq9} shows on the left an example of an anti-king but not anti-queen Latin square. It is formed by shifting the first row by 2 to the right. The diagonal stretching from the bottom left to the top right repeats the pattern 963, thus these numbers see each other by queen's move. The example of non-consecutive by king's move Latin square in Figure~\ref{sec:NonconbyKM} is also an example of an anti-king Latin square that is not an anti-queen Latin square, as the main diagonal stretching from the bottom left to the top right repeats the pattern 693. On the right in Figure~\ref{fig:akNotaq9}
we see a smaller example of order 6. This square is not cyclic. It is formed by cycling the previous row by 2 to the left, except the fourth row is formed by the cycling of the third row by 3.
\begin{figure}[ht!]
{
	\begin{center}
	\begin{tikzpicture}
	\draw[step=0.5cm,color=black, line width=1.5] (-2.5,-2.5) grid (2,2);
	\node at (-2.25, +1.75){1};
	\node at (-2.25, +1.25){8};
	\node at (-2.25, +0.75){6};
	\node at (-2.25, +0.25){4};
	\node at (-2.25, -0.25){2};
	\node at (-2.25, -0.75){9};
	\node at (-2.25, -1.25){7};
	\node at (-2.25, -1.75){5};
	\node at (-2.25, -2.25){3};
	\node at (-1.75, +1.75){2};
	\node at (-1.75, +1.25){9};
	\node at (-1.75, +0.75){7};
	\node at (-1.75, +0.25){5};
	\node at (-1.75, -0.25){3};
	\node at (-1.75, -0.75){1};
	\node at (-1.75, -1.25){8};
	\node at (-1.75, -1.75){6};
	\node at (-1.75, -2.25){4};
	\node at (-1.25, +1.75){3};
	\node at (-1.25, +1.25){1};
	\node at (-1.25, +0.75){8};
	\node at (-1.25, +0.25){6};
	\node at (-1.25, -0.25){4};
	\node at (-1.25, -0.75){2};
	\node at (-1.25, -1.25){9};
	\node at (-1.25, -1.75){7};
	\node at (-1.25, -2.25){5};
	\node at (-0.75, +1.75){4};
	\node at (-0.75, +1.25){2};
	\node at (-0.75, +0.75){9};
	\node at (-0.75, +0.25){7};
	\node at (-0.75, -0.25){5};
	\node at (-0.75, -0.75){3};
	\node at (-0.75, -1.25){1};
	\node at (-0.75, -1.75){8};
	\node at (-0.75, -2.25){6};
	\node at (-0.25, +1.75){5};
	\node at (-0.25, +1.25){3};
	\node at (-0.25, +0.75){1};
	\node at (-0.25, +0.25){8};
	\node at (-0.25, -0.25){6};
	\node at (-0.25, -0.75){4};
	\node at (-0.25, -1.25){2};
	\node at (-0.25, -1.75){9};
	\node at (-0.25, -2.25){7};
	\node at (0.25, +1.75){6};
	\node at (0.25, +1.25){4};
	\node at (0.25, +0.75){2};
	\node at (0.25, +0.25){9};
	\node at (0.25, -0.25){7};
	\node at (0.25, -0.75){5};
	\node at (0.25, -1.25){3};
	\node at (0.25, -1.75){1};
	\node at (0.25, -2.25){8};
	\node at (0.75, +1.75){7};
	\node at (0.75, +1.25){5};
	\node at (0.75, +0.75){3};
	\node at (0.75, +0.25){1};
	\node at (0.75, -0.25){8};
	\node at (0.75, -0.75){6};
	\node at (0.75, -1.25){4};
	\node at (0.75, -1.75){2};
	\node at (0.75, -2.25){9};
	\node at (1.25, +1.75){8};
	\node at (1.25, +1.25){6};
	\node at (1.25, +0.75){4};
	\node at (1.25, +0.25){2};
	\node at (1.25, -0.25){9};
	\node at (1.25, -0.75){7};
	\node at (1.25, -1.25){5};
	\node at (1.25, -1.75){3};
	\node at (1.25, -2.25){1};
	\node at (1.75, +1.75){9};
	\node at (1.75, +1.25){7};
	\node at (1.75, +0.75){5};
	\node at (1.75, +0.25){3};
	\node at (1.75, -0.25){1};
	\node at (1.75, -0.75){8};
	\node at (1.75, -1.25){6};
	\node at (1.75, -1.75){4};
	\node at (1.75, -2.25){2};
	\end{tikzpicture}
\quad \quad
\begin{tikzpicture}
\draw[step=0.5cm,color=black, line width=1.5] (-1,-1) grid (2,2);
\node at (-0.75, +1.75) {1};
\node at (-0.25, +1.75) {2};
\node at (+0.25, +1.75) {3};
\node at (+0.75, +1.75) {4};
\node at (+1.25, +1.75) {5};
\node at (+1.75, +1.75) {6};
\node at (-0.75, +1.25) {3};
\node at (-0.25, +1.25) {4};
\node at (+0.25, +1.25) {5};
\node at (+0.75, +1.25) {6};
\node at (+1.25, +1.25) {1};
\node at (+1.75, +1.25) {2};
\node at (-0.75,+0.75) {5};
\node at (-0.25,+0.75) {6};
\node at (+0.25,+0.75) {1};
\node at (+0.75,+0.75) {2};
\node at (+1.25,+0.75) {3};
\node at (+1.75, +0.75) {4};
\node at (-0.75,+0.25) {2};
\node at (-0.25,+0.25) {3};
\node at (+0.25,+0.25) {4};
\node at (+0.75,+0.25) {5};
\node at (+1.25,+0.25) {6};
\node at (+1.75, +0.25) {1};
\node at (-0.75, -0.25) {4};
\node at (-0.25, -0.25) {5};
\node at (+0.25, -0.25) {6};
\node at (+0.75, -0.25) {1};
\node at (+1.25, -0.25) {2};
\node at (+1.75, -0.25) {3};
\node at (-0.75,-0.75) {6};
\node at (-0.25,-0.75) {1};
\node at (+0.25,-0.75) {2};
\node at (+0.75,-0.75) {3};
\node at (+1.25,-0.75) {4};
\node at (+1.75, -0.75) {5};
\end{tikzpicture}
	\end{center}
	\caption{Anti-king, but not anti-queen Latin squares of orders 9 and 6.}
	\label{fig:akNotaq9}
	}
\end{figure}

Now we show that anti-king, but not anti-queen Latin squares exist for composite large $n$.

\begin{theorem}
An anti-king, but not anti-queen Latin square exists for composite $n > 6$.
\end{theorem}

\begin{proof}
We construct a cyclic anti-king Latin square. The shift $k$ has to be different from plus or minus one to make it anti-king. Moreover, it has to be coprime with $n$, so it is a Latin square. Also, $k-1$ or $k+1$ has to have a common divisor with $n$ so that one of the diagonals has repeated numbers. 

If $n$ is even, we take a shift $k$ between 2 and $n-2$ and coprime with $n$. Thus, it will be a Latin square, and since $k\pm 1$ are even, it will not be anti-queen, and since it is not $\pm 1$, it is anti-king. 

For odd $n$, let $p$ be the smallest prime factor. Then $p+1$ is coprime with $n$ since it is coprime with $p$ and all larger primes. So using $k=p+1$ as a shift creates a Latin square. Also $1 < k < n-1$, so the square is anti-king. We also know that $k-1$ is a proper divisor of $n$, and the resulting square is not anti-queen.
\end{proof}

By construction in the theorem above, we see that we can't construct a cyclic anti-king Latin square that is not anti-queen for prime orders.

\section{Chess Latin squares}\label{sec:Chess}

For a chess piece chiece, let us call a \textit{chiece} Latin square a Latin square such that for each cell in a square, there is a cell containing the same number a chiece's move apart. A rook Latin square is an oxymoron: a Latin square does not allow any numbers to repeat in the same row or column, but a rook Latin square requires numbers to repeat in the same row or column. As before, queen Latin squares are the same as bishop Latin squares; we call them \textit{bishop} Latin squares.

Trivially, chiece Latin squares of order 1 do not exist.

\subsection{Knight Latin squares}

In a \textit{knight} Latin square, for each number in a particular cell, the same number exists in a cell which is a knight's move apart.

For $n=1$ or 2, a knight Latin square is impossible: no two cells are a knight's move apart. For $n=3$, it is also impossible: no cell is a knight's move apart from the center cell.

The lexicographically earliest 4 by 4 knight Latin square is the Latin square we previously saw in Figure~\ref{fig:LC4}. Interestingly, swapping the last two rows yields another knight Latin square.

Knight Latin squares are similar to strict knight's move Latin squares mentioned in Section~\ref{sec:intro}. The difference is that in strict knight's move Latin squares, any two cells which contain the same number are joined by a sequence of knight's moves using cells containing that number. In contrast, in knight Latin squares, for each number in a square cell, there is the same number a knight's move apart, but two cells with the same number do not have to be joined by a sequence of knight moves. Hence, all strict knight's move Latin squares are knight Latin squares, but not all knight Latin squares are strict knight's move Latin squares.

A rather large knight Latin square that is not a strict knight's move Latin square is shown in Figure~\ref{fig:kLsNotskLs}.

\begin{figure}[ht!]
{
    \begin{center}
        \begin{tikzpicture}
            \draw[step=0.5cm,color=black, line width=1.5] (0,0) grid (5,5);
            \node at (0.25, 4.75) {1};
            \node at (0.75, 4.75) {2};
            \node at (1.25, 4.75) {3};
            \node at (1.75, 4.75) {4};
            \node at (2.25, 4.75) {5};
            \node at (2.75, 4.75) {6};
            \node at (3.25, 4.75) {7};
            \node at (3.75, 4.75) {8};
            \node at (4.25, 4.75) {9};
            \node at (4.75, 4.75) {10};
            
            \node at (0.25, 4.25) {5};
            \node at (0.75, 4.25) {4};
            \node at (1.25, 4.25) {1};
            \node at (1.75, 4.25) {3};
            \node at (2.25, 4.25) {2};
            \node at (2.75, 4.25) {10};
            \node at (3.25, 4.25) {9};
            \node at (3.75, 4.25) {6};
            \node at (4.25, 4.25) {8};
            \node at (4.75, 4.25) {7};
            
            \node at (0.25, 3.75) {4};
            \node at (0.75, 3.75) {3};
            \node at (1.25, 3.75) {2};
            \node at (1.75, 3.75) {5};
            \node at (2.25, 3.75) {1};
            \node at (2.75, 3.75) {9};
            \node at (3.25, 3.75) {8};
            \node at (3.75, 3.75) {7};
            \node at (4.25, 3.75) {10};
            \node at (4.75, 3.75) {6};
            
            \node at (0.25, 3.25) {2};
            \node at (0.75, 3.25) {5};
            \node at (1.25, 3.25) {4};
            \node at (1.75, 3.25) {1};
            \node at (2.25, 3.25) {3};
            \node at (2.75, 3.25) {7};
            \node at (3.25, 3.25) {10};
            \node at (3.75, 3.25) {9};
            \node at (4.25, 3.25) {6};
            \node at (4.75, 3.25) {8};
            
            \node at (0.25, 2.75) {3};
            \node at (0.75, 2.75) {1};
            \node at (1.25, 2.75) {5};
            \node at (1.75, 2.75) {2};
            \node at (2.25, 2.75) {4};
            \node at (2.75, 2.75) {8};
            \node at (3.25, 2.75) {6};
            \node at (3.75, 2.75) {10};
            \node at (4.25, 2.75) {7};
            \node at (4.75, 2.75) {9};
            
            \node at (0.25, 2.25) {6};
            \node at (0.75, 2.25) {7};
            \node at (1.25, 2.25) {8};
            \node at (1.75, 2.25) {9};
            \node at (2.25, 2.25) {10};
            \node at (2.75, 2.25) {1};
            \node at (3.25, 2.25) {2};
            \node at (3.75, 2.25) {3};
            \node at (4.25, 2.25) {4};
            \node at (4.75, 2.25) {5};
            
            \node at (0.25, 1.75) {10};
            \node at (0.75, 1.75) {9};
            \node at (1.25, 1.75) {6};
            \node at (1.75, 1.75) {8};
            \node at (2.25, 1.75) {7};
            \node at (2.75, 1.75) {5};
            \node at (3.25, 1.75) {4};
            \node at (3.75, 1.75) {1};
            \node at (4.25, 1.75) {3};
            \node at (4.75, 1.75) {2};
            
            \node at (0.25, 1.25) {9};
            \node at (0.75, 1.25) {8};
            \node at (1.25, 1.25) {7};
            \node at (1.75, 1.25) {10};
            \node at (2.25, 1.25) {6};
            \node at (2.75, 1.25) {4};
            \node at (3.25, 1.25) {3};
            \node at (3.75, 1.25) {2};
            \node at (4.25, 1.25) {5};
            \node at (4.75, 1.25) {1};
        
            \node at (0.25, 0.75) {7};
            \node at (0.75, 0.75) {10};
            \node at (1.25, 0.75) {9};
            \node at (1.75, 0.75) {6};
            \node at (2.25, 0.75) {8};
            \node at (2.75, 0.75) {2};
            \node at (3.25, 0.75) {5};
            \node at (3.75, 0.75) {4};
            \node at (4.25, 0.75) {1};
            \node at (4.75, 0.75) {3};
            
            \node at (0.25, 0.25) {8};
            \node at (0.75, 0.25) {6};
            \node at (1.25, 0.25) {10};
            \node at (1.75, 0.25) {7};
            \node at (2.25, 0.25) {9};
            \node at (2.75, 0.25) {3};
            \node at (3.25, 0.25) {1};
            \node at (3.75, 0.25) {5};
            \node at (4.25, 0.25) {2};
            \node at (4.75, 0.25) {4};
        \end{tikzpicture}
    \end{center}
	\caption{A knight Latin square that is not a strict knight's move Latin square.}
	\label{fig:kLsNotskLs}
}
\end{figure}

Interestingly, this square is the product of the 5 by 5 strict knight's move Latin square shown on the left of Figure~\ref{fig:skm5} and a 2 by 2 Latin square.

We can prove that all 5 by 5 knight Latin squares are strict knight's move Latin squares.

\begin{proposition}
All knight Latin squares of size 5 are also strict knight's move Latin squares.
\end{proposition}

\begin{proof}
Suppose a knight Latin square that is not a strict knight's move Latin square exists in order 5. Then, at least one number is broken into two chains, of 3 and 2 cells, that can be traversed with knight's moves such that no cell in the first chain is a knight's move apart from any cell in the second chain.

The chain of two cells must occupy either two consecutive rows or two consecutive columns. Without loss of generality, we can assume that the 2-chain occupies two consecutive rows. These two rows can't be in the middle, as otherwise, they will break the 3-chain, so without loss of generality, we can assume that the 2-chain occupies the last two rows.

Let $x$ represent the cells that are part of the 3-chain and $y$ represent possible leftover placements for the cells in the 2-chain. Up to reflection with respect to the middle vertical line, all possibilities are shown in Figure~\ref{fig:2-3chains}.
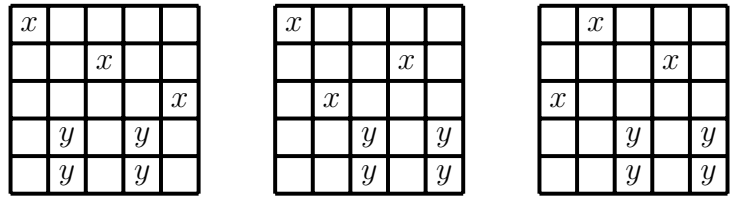
\begin{figure}[ht!]
	\begin{center}
		\begin{tikzpicture}
	\draw[step=0.5cm,color=black, line width=1.5] (0,0) grid (2.5,2.5);
    \node at (0.25,2.25) {$x$};
	\node at (1.25,1.75) {$x$};
	\node at (2.25,1.25) {$x$};
	\node at (0.75,0.75) {$y$};
	\node at (1.75,0.75) {$y$};
	\node at (0.75,0.25) {$y$};
	\node at (1.75,0.25) {$y$};
	\end{tikzpicture}
\quad \quad 
	\begin{tikzpicture}
	\draw[step=0.5cm,color=black, line width=1.5] (0,0) grid (2.5,2.5);
    \node at (0.25,2.25) {$x$};
	\node at (1.75,1.75) {$x$};
	\node at (0.75,1.25) {$x$};
	\node at (1.25,0.75) {$y$};
	\node at (2.25,0.75) {$y$};
	\node at (1.25,0.25) {$y$};
	\node at (2.25,0.25) {$y$};
	\end{tikzpicture}
\quad \quad 
	\begin{tikzpicture}
	\draw[step=0.5cm,color=black, line width=1.5] (0,0) grid (2.5,2.5);
    \node at (0.75,2.25) {$x$};
	\node at (1.75,1.75) {$x$};
	\node at (0.25,1.25) {$x$};
	\node at (1.25,0.75) {$y$};
	\node at (2.25,0.75) {$y$};
	\node at (1.25,0.25) {$y$};
	\node at (2.25,0.25) {$y$};
	\end{tikzpicture}
	\caption{Placements of a 3-chain with leftover options for a 2-chain}
\label{fig:2-3chains}
	\end{center}
\end{figure}
Within the cells marked $y$, a 2-chain with the cells a knight's move apart can exist, but there will always be at least one cell a knight's move apart from at least one $y$ cell in the 3-chain. Then both chains are connected by a knight's move, and this is a strict knight's move Latin square. This contradicts with our original assumption, therefore all size-5 knight Latin squares are strict knight's move Latin squares.
\end{proof}

\subsection{Bishop Latin squares}

For $n=2$, any Latin square is a bishop Latin square. But the opposite is true for $n=3$: no bishop Latin square exists.

\begin{theorem}
Bishop Latin squares exist for even orders.
\end{theorem}

\begin{proof}
We prove this by showing that all left-cyclic or right-cyclic Latin squares of even order are bishop Latin squares. Without loss of generality, assume that the square is right-cyclic. Then each number is repeated in a broken (wrapped-around) diagonal, so we only need to ensure that the two numbers in the top right corner and the bottom left corner see the same numbers on the main anti-diagonal. The numbers in the anti-diagonal are shifted by 2. This means that, for an even-ordered square, the diagonal will repeat after $\frac{n}{2}$ cells, so the numbers in both corners are repeated in the main anti-diagonal a bishop's move away. Thus, all left-cyclic or right-cyclic Latin squares of even order are bishop Latin squares.
\end{proof}

Before discussing odd orders, we make the following observation. Suppose, similar to a checkerboard, the Latin square is alternatively colored black and white. Assume the bottom-left corner is black. Then for odd orders, the main diagonals are black. Now consider a particular entry in the Latin square. If it is on a white cell, then the sum of the coordinates of this entry is odd; and if it is on a black cell, then the sum of the coordinates of this entry is even. If we sum the coordinates of a particular number in a Latin square, we get an even number, $n(n+1)$. This means that each number is placed on white cells an even number of times. This observation allows us to produce an elegant proof of why order 5 bishop Latin squares do not exist.

\begin{proposition}
Bishop Latin squares of order 5 do not exist.
\end{proposition}

\begin{proof}
Assume, for the sake of contradiction, that an order 5 bishop Latin square exists. Then each number is a bishop's move apart from the same number, only one instance of the particular number cannot be placed on a particular color. Also, from the above observation, each number must be placed on the white cells an even number of times. Thus, the only possibility is for each number to occupy 2 white cells and 3 black cells, implying that the Latin square has 10 white cells and 15 black cells. But this contradicts with the Latin square having 12 white cells and 13 black cells, so the original assumption that bishop Latin squares of order 5 exist is false.
\end{proof}

But bishop Latin squares do exist for larger odd orders.

\begin{theorem}
If $n$ is a product of two distinct odd primes, then there exists a bishop Latin square of order $n$.
\end{theorem}

\begin{proof}
Let us denote the smaller prime factor of $n$ as $p$, the larger as $q$, and the average of them as $c$. Note that $c$ is an integer that is coprime with both $p$ and $q$. Let the first row start with 1, increasing by $c-p$ and taking the result modulo $n$ when necessary. We construct each next row from the previous row by adding $c$ and taking the result modulo $n$. Once the square is created, we add 1 to all of the numbers so that the square includes $n$ and does not include 0.

As $c-p$ is coprime with $n$, each row contains all the numbers once. Similarly, as $c$ is coprime with $n$, each column contains every number once. Thus we constructed a Latin square. Now we look at diagonals. Consider a number $x$. The next number in the SE direction is $x+ 2c-p = x+q$ modulo $n$ and the next number in the SW direction is $x -c+p+c = x+p$ modulo $n$. Thus, on the diagonals in the SE direction, each number repeats itself every $p$ rows, while in the other diagonal direction every number repeats itself every $q$ rows. As both $p$ and $q$ are greater than 2 and each cell is on a diagonal at least $\frac{n}{2}$ cells long, each number must repeat itself in a diagonal which it is on. Thus, each number is a bishop's move away from the same number.
\end{proof}

Let \textit{strict bishop's move Latin squares} be Latin squares where all the cells with the same number are connected by a chain of bishop's moves.

\begin{theorem}
For odd orders greater than 1, strict bishop's move Latin squares do not exist. For even orders, they exist.
\end{theorem}

\begin{proof}
For odd orders, we color the board in checkerboard colors as we have previously done. Then all the cells with the same numbers must be placed on the same color. Thus, the number of cells of the same color must be divisible by $n$. A Latin square of odd order has $\frac{n^2+1}{2}$ cells of one color and $\frac{n^2-1}{2}$ cells of another. This creates a contradiction as $\frac{n^2+1}{2}$ is not divisible by $n$ for $n > 1$. Thus, strict bishop's move Latin squares of an odd order that is not 1 do not exist.

For even orders, the 1-cyclic Latin square is a strict bishop's move Latin square. The cells not in the top-right or bottom-left corner repeat in the SE and NW direction. The cells in the top-right and bottom-left corner repeat in the diagonal between them at a distance $\frac{n}{2}$.
\end{proof}

\subsection{King Latin squares of odd sizes}

For $n=2$, any Latin square is a king Latin square. But the opposite is true for $n=3$: no king Latin square exists.

\begin{theorem}
King Latin squares of odd sizes do not exist.
\end{theorem}

\begin{proof}
Size 1 King Latin squares do not exist. And for larger sizes, it is enough to show that it is impossible to arrange the first 2 rows such that a number in the first row repeats a king's move apart.

Without loss of generality, assume that the first row contains the numbers 1 to $n$ in order. Suppose $n=3$. Then both numbers 1 and 3 need to be in the center of the second-row, creating a contradiction.

Suppose $n$ is an odd number greater than 3. Then 1 must be placed in the second cell of the first row and $n$ in the second to last cell of the second row. Each number in the first row needs to see (by king's move) another number of the same value in the second row. That means each number in the second row needs to be seen from a number in the first row. That means we have to place 2 in the first cell of the second row and $n-1$ in the last cell. That means numbers $3$ through $n-2$ need to be placed in the cells numbered 3 through $n-2$ in the second row. It follows that if we can create two rows for $n$ such that every cell in the first row sees the same value by king's move, then we can create two rows for $n-2$. Now we can invoke induction. Given that the king Latin square of order 3 doesn't exist, it follows that it doesn't exist for any odd $n$. 
\end{proof}

\subsection{Larger sizes}

We discuss larger sizes for all chieces, in one theorem.

\begin{theorem}\label{thm:ChessLargerSizes}
If an $n$ by $n$ chiece Latin square exists, then for all positive integers $m$, there exists an $nm$ by $nm$ chiece Latin square.
\end{theorem}
\begin{proof}
Consider an $n$ by $n$ chiece Latin square $C$ and any  $M$: any $m$ by $m$ Latin square $M$. The product of $M$ and $C$ is another square that consists of $n$ by $n$ blocks. Each element within any block is equal to the corresponding element in $C$ incremented by the same number. So each chiece can find another chiece a chiece's move apart in its $n$ by $n$ block. Thus, this construction produces a valid chiece Latin square of size $nm$.
\end{proof}

\begin{corollary}
Bishop and king Latin squares exist for all even orders, and knight Latin squares exist for all orders divisible by four or five.
\end{corollary}

It follows that toroidal chiece Latin squares also exist in the corresponding sizes. For example, using the construction in Theorem~\ref{thm:ChessLargerSizes}, we can generate a king and bishop Latin square of order 4 from such a square of order 2. As a result, we get the lexicographically earliest Latin square of order 4 shown in Figure~\ref{fig:LE4}.

\section{Acknowledgements}

This project was done as part of MIT PRIMES STEP, a program that allows students in grades 7 through 9 to try research in mathematics. Tanya Khovanova is the mentor of this project. We are grateful to PRIMES STEP for this opportunity.

\end{document}